\documentclass[12pt]{article}

\usepackage[utf8]{inputenc}
\usepackage[english]{babel}

\usepackage{multirow}

\usepackage{amsmath,amssymb,enumerate,amsthm,flafter}

\usepackage{tikz}

\usetikzlibrary{calc,through,backgrounds,positioning,shapes.misc}
\usetikzlibrary{arrows,decorations.pathmorphing,fit,petri}

\newcommand{\eme}[1]{}

\title{Combinatorial Relationship Between Finite Fields and Fixed Points of Functions Going Up and Down}

\author{
  Emerson Le\'on\\
  \texttt{emersonleon@gmail.com}
  \and
  Juli\'an Pulido\\
  \texttt{jd.pulido@uniandes.edu.co}
}

\newcommand{\R}{{\ensuremath{\mathbb{R}}}}

\newcommand{\F}{{\ensuremath{\mathbb{F}}}}

\DeclareMathOperator{\FP}{FP}

\theoremstyle{plain}
\newtheorem{theorem}{Theorem}[section]
\newtheorem{proposition}[theorem]{Proposition}
\newtheorem{corollary}[theorem]{Corollary}

\newtheorem*{theorem*}{Theorem}

\theoremstyle{definition}
\newtheorem{definition}[theorem]{Definition}
\newtheorem{notation}[theorem]{Notation}

\newtheorem{example}[theorem]{Example}
\begin{document}

\maketitle

\abstract{We explore a combinatorial bijection between two seemingly unrelated topics: the roots of irreducible polynomials  of degree $m$ over a finite field $\F_p$ for a prime number $p$ and the number of points that are periodic of order $m$ for a continuous piece-wise linear function  $g_p:[0,1]\rightarrow[0,1]$ that \emph{goes up and down $p$ times} with slope $\pm p$.  We provide a bijection between $\F_{p^n}$ and the fixed points of $g^n_p$ that naturally relates some of the structure in both worlds. Also we extend our result to other families of continuous functions that go up and down $p$ times, in particular to  Chebyshev polynomials, where we get a better understanding of its fixed points. A generalization for other piece-wise linear functions that are not necessarily continuous is also provided.
}

\section{Introduction}

Let us consider the  family of functions $g_{p}:[0,1]\rightarrow [0,1]$ that linearly increases from $0$ to $1$ on $I_k$   for $k$ even,  and linearly decreases from $1$ to $0$ on $I_k$   for $k$ odd, where $I_k:=\left[\frac{k}{p},\frac{k+1}{p}\right]$ for $0\le k\le p-1$. They can be described as follows.

\begin{definition}\label{defg}
 The function  $g_{p}:[0,1]\rightarrow [0,1]$ is given by
\[g_{p}(x)=
\begin{cases} 
{p}x-k, &\text{for $\frac{k}{{p^n}}\leq x\leq \frac{k+1}{{p^n}}$ with $k$  even.}\\
k+1-{p}x, &\text{for $\frac{k}{{p^n}}\leq x\leq \frac{k+1}{{p^n}}$ with $k$ odd.}
\end{cases} \]
\end{definition}

Denote by $g_p^n$ to the function obtained by taking  $g_p$ composed with itself $n$ times. A point $x$ is called to be \emph{periodic of order $m$} if $g_p^m(x)=x$, and $g_p^i(x)\neq x$ for $i=1,\,2,\ldots,m-1$.

\begin{figure}[h]\centering
\begin{picture}(205,90)
\multiput(0,0)(120,0){2}{\vector(0,1){85}}
\multiput(0,0)(120,0){2}{\vector(1,0){85}}
\put(40,80){\line(1,-2){40}}
\put(0,0){\line(1,2){40}}
\multiput(140,80)(40,0){2}{\line(1,-4){20}}
\multiput(120,0)(40,0){2}{\line(1,4){20}}
\linethickness{2pt}
\multiput(0,0)(120,0){2}{\line(1,1){80}}
\end{picture}
\caption{Fixed points of $g_2$ and $g_2^2$.}\label{f}
\end{figure}
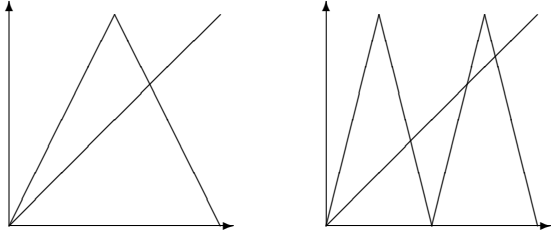

It is easy to check that $g_p^n$ is a function going up and down $p^n$ times, and moreover $g_p^n=g_{p^n}$. Therefore $g_p^n$ has $p^n$ fixed points and we can compute them by solving the corresponding linear equations (see Proposition \ref{gfixedpts}).
Let $\FP(g_p^n)$ be the set of all fixed points of $g_p^n$.

Notice that if $x\in \FP(g_p^n)$ then $x$ is  periodic of order $d$ for some divisor $d$ of $n$.  The study of these fixed points and their orbits by $g_p$ is important in chaos theory and dynamical systems (see \cite{chaos}). These functions are particularly interesting as examples for the Sharkovsky Theorem (see \cite{sharkovsky}), that implies here the existence of points of $g_p$ of order $m$ for any integer $m\ge 1$.

We are specially interested in the case when $p$ is a prime number, since we will relate fixed points of $g_p^n$  with the elements of the finite field $\F_{p^n}$ with $p^n$ elements, for $p$ prime.

We begin this work with the following combinatorial observation.

\begin{theorem}\label{t1}
Let $p$ be a prime number and $m\ge 1$. The number of points $x\in \FP(g_p^m)$ that are periodic of order $m$ is equal to $m$ times the number  of monic irreducible polynomials of degree $m$ over the field $\F_p$.
\end{theorem}

The number of irreducible polynomials is well known (see \cite{ffields}) and the recurrence relation that it satisfies is exactly the same as the one used to count  orbits of fixed points of $g_p$. We explain this in Section \ref{combarg}. We  conclude the following.

\begin{corollary} \label{cor1}
 There is a bijection $B:\FP(g_p^n)\rightarrow \F_{p^n}$, such that $B(x)^p=B(g_p(x))$ for every $x\in \FP(g_p^n)$. 
\end{corollary}

The existence of this bijection is a simple consequence of the enumerative relationship previously described (Theorem \ref{t1}).  In this paper we provide a  bijection that also relates with the multiplicative structure of $\F_{p^n}$. 

We construct a permutation map $\pi_{p^n}:\{0,\,1,\ldots,\,p^n-1\}\rightarrow \{0,\,1,\ldots,\,p^n-1\}$  that will help us to create our  bijection $B$. 
\begin{definition}\label{pipn}
The permutation  $\pi_{p^n}$ is defined recursively on $n$ 
as follows: take $\pi_{p^1}$ to be the identity map from 0 to $p-1$. Then to define $\pi_{p^n}(k)$ for $0\le k\le p^n-1$, consider the value $a$ such that $a p^{n-1}\le k<(a+1)p^{n-1}$ and then take
 \[\pi_{p^n}(k)=\begin{cases}
  \pi_{p^{n-1}}(k-ap^{n-1}) + ap^{n-1}\text{ for $a$ even}\\
 \pi_{p^{n-1}}(p^{n-1}-(k-ap^{n-1})-1) + ap^{n-1}\text{ for $a$ odd}
\end{cases}\]
\end{definition}

\begin{example}
Some examples for $p=2$ and $p=3$ of how $\pi_{p^n}$  permute the numbers $0, 1, \dots ,p^n-1$ :

$\pi_{2^1}: 0, 1$

$\pi_{2^2}: 0, 1, 3, 2$

$\pi_{2^3}: 0, 1, 3, 2, 6, 7, 5, 4$

$\pi_{2^4}: 0, 1, 3, 2, 6, 7, 5, 4, 12, 13, 15, 14, 10, 11, 9, 8$

$\pi_{3^1}: 0, 1, 2$

$\pi_{3^2}: 0, 1, 2, 5, 4, 3, 6, 7, 8$

$\pi_{3^3}: 0, 1, 2, 5, 4, 3, 6, 7, 8, 17,  16, 15,12, 13, 14, 11, 10, 9,
18, 19, 20, 23, 22, 21, 24, 25, 26$
\end{example}

\begin{definition}[Bijection $B_\alpha$]
 Let $\alpha$ be a primitive root (i.e., a generator of the multiplicative group) of $\F_{p^n}$ (\cite{ffields}) and  let $0=x_0<x_1<\cdots<x_{p^{n}-1}$ be the fixed points of $g_p^n$. 
 We define the bijection $B_\alpha:\FP(g_p^n)\rightarrow \F_{p^n}$ by  
 $B_\alpha(x_{0})=0\in \F_{p^n}$ or  $B_\alpha(x_k)=\alpha^{\pi_{p^n}(k)}$ for all other $k>0$. 
\end{definition}

\begin{theorem}\label{t2}
The function $B_\alpha$ is a bijection such that $(B_\alpha(x_k))^p=B_\alpha(g_p(x_k))$ for any fixed point $x_k\in \FP(g_p^n).$ 
\end{theorem}


We begin in Section \ref{combarg} with the proof of Theorem \ref{t1} by  obtaining the explicit counting in both cases using the M\"obius inversion formula.
Then in Section \ref{sbasep} we use the representation in base $p$ of $g_p$ and of $\pi_{p^n}$ in order to prove Theorem \ref{t2}  and Theorem \ref{t3}. We need to consider two different cases, since the arguments are slightly different depending on whether $p=2$ or $p$ is odd.

In Section \ref{scheb} we use a continuity argument to extend our result to  other functions $f_p$ going up and down $p$ times, in particular to Chebyshev polynomials of the first kind $T_p(x)$ on the interval $[-1,1]$. We  study its fixed points and extend our bijection with $\F_{p^n}$ to $\FP(T_{p^n})$. 
Finally in Section \ref{sgeneralud} we generalize Theorem \ref{t2} and most of the results of previous sections to other similar piece-wise linear functions $g_{p,I}: [0,1]\rightarrow[0,1]$ that are not necessarily continuous,  going up or down with slope $\pm{p}$ with increasing pattern given by a set $I\subseteq \{0,1, \ldots, p-1\}$ (see definition \ref{defgI}). We define the corresponding permutations $\pi_{p^n,I}$ analogous to Definition \ref{pipn} in Definition \ref{pipnI} and we use them in Theorem \ref{t2I} to generalize Theorem \ref{t2} to  bijections $B_{\alpha,I}$ from the fixed points of $g^n_{p,I}$ to the finite field $\F_{p^n}$, for each $I$.

\subsection*{Acknowledgments}
This paper is dedicated to two former professors at Universidad Nacional de Colombia that are no longer with us, Yu Takeuchi and Alexander Zavadsky, that talked about these seemingly unrelated topics in different lectures on the same semester. Also special thanks to  Alexander Fomin,  Felipe Rincón and Tristram Bogart  for fruitful conversations.

 


\section{The combinatorial relationship}\label{combarg}
\begin{proof}[Proof of Theorem \ref{t1}]
 First let $I_p(n)$ be the number of monic irreducible polynomials of degree $n$ over a finite field $\F_p$ for $p$ prime. This can be explicitly computed using the \emph{Möbius inversion formula}, and the fact that 
$$\sum_{d\mid n}dI_p(d)=p^n,$$ 
that comes from the degrees in the irreducible factorization of the polynomial $x^{p^n}-x$ over $\F_p$ (see \cite{ffields}). 

Precisely the same recursive relation is satisfied by our fixed points. Notice that all fixed points in $\FP(f_p^n)$ are of order $d$ so that $d$ divides $n$, and then if we denote by $J(m)$ the number of fixed points in $\FP(f_p^m)$ of order $m$ we get:
$$\sum_{d\mid n}J_p(d)=p^n.$$

We conclude that $J_p(m)=mI_p(m)$, since the \emph{Möbius inversion formula} in both case provides us the same value
$$J_p(m)=mI_p(m)=\sum_{d\mid n}\mu(n/d)p^d,$$
where $\mu$ is the classical Möbius function for the natural numbers (with respect to divisibility). 
\end{proof}

\section{Fixed points and  base $p$ representation}\label{sbasep}

\begin{proposition}\label{gfixedpts}
Let $x_0<x_1<\ldots <x_{p^n-1}$ be the fixed points of the function $g_p^n$.
Then, for $k<p^n$ we have that 
 $x_k=\frac{k}{p^n-1}$ if $k$ is even,  or 
 $x_k=\frac{k+1}{p^n+1}$ if $k$ is odd. 
\end{proposition}

\begin{proof}
Notice that  $\frac{k}{{p^n}}\leq x_k\leq \frac{k+1}{{p^n}}$, then if  $k$ is even, we have that $g_p^n(x_k) = {p^n}x_k-k = x_k $ which implies $x_k=\frac{k}{p^n-1}$, if $k$ is odd then   $g_p^n(x_k) = k+1-{p^n}x_k = x_k $ which implies $x_k=\frac{k+1}{p^n+1}$
\end{proof}

\begin{notation}
We denote the expression for a number $k$ in base $p$ as $k=(a_1a_2\ldots a_n)_p$, meaning by this that  $$k=a_1p^{n-1}+a_2p^{n-2}+\cdots +a_{n-1}p+a_n.$$ For other rational values, we write $(0.a_1a_2\ldots)_p$ for the number $\sum_{i\geq1} a_i p^{-i}$. If the expression is periodic, the periodic part is overlined. (For convenience, we sometimes don't use the minimal periodic expression for a number, but some repetitions of it.)
\end{notation}

\begin{proposition}\label{periodicfp} The expression in base $p$ of $x=\frac{k}{p^n-1}$ is periodic, where the first $n$ digits represent $k$ in base $p$ and $a_{n+i}=a_i$ for all $i\ge 1$. 
Also, if $x=\frac{k+1}{p^n+1},$ its expression in base $p$ has period $2n$, where the first $n$ digits represent $k$ in base $p$, while the next $n$ digits are complementary of the first $n$ digits, that is  $a_{n+i}=p-1-a_i$.
\end{proposition}
\begin{proof} 
Let $x = \sum_{i\geq1} a_i p^{-i}$ the expression for $x$ in base $p$. Lets also  express $k$ in base $p$. $$k = k_1p^{n-1}+k_2p^{n-2}+\cdots +k_{n-1}p+k_n.$$
If $\frac{k}{p^n-1} = x= \sum_{i\geq 1} a_i p^{-i}$, then $(p^n-1)x = k$ thus
$$k=(p^n - 1) \sum_{i\geq 1} a_{i}p^{-i} = \sum_{i\geq 1}  a_i p^{n-i} - \sum_{i\geq 1} a_i p^{-i},$$
$$k=\sum_{i= 1}^{n}  a_i p^{n-i} +  \sum_{j\geq 1}  a_{j+n} p^{-j} - \sum_{i\geq 1} a_i p^{-i}.$$ 
In the previous line we replaced $i=j+n$ for $j\ge 1$. 
If we look at the coefficients of the non-negative powers of $p$ (where $1-n\le j\le 0$) we obtain that  $k_{i} = a_{i}$ for $1\leq i \leq n$. Looking at coefficients with negative powers of $p$ (that are zero in the right expression), we find that $a_{n+i} = a_i$ for all $i\geq 1$,  which proves the first part of the proposition.

Similarly if $x = \frac{k+1}{p^n+1}$ then $$(p^n+1)\sum_{i\geq1} a_i p^{-i} = \sum_{i\geq1} \left(a_i p^{n-i} + a_i p^{-i}\right) = k+1,$$ and therefore $k=\sum_{i\geq1}\left( a_i p^{n-i} + a_i p^{-i}\right) -1 $. Observe that -1 can be expressed as the telescoping series  $-1= \sum_{i\geq 1} \left(p^{-i} - p^{-(i-1)}\right),$ and then $$k = \sum_{i\geq1} \left(a_i p^{n-i} + a_i p^{-i} + p^{-i} - p^{-i+1}\right) = \sum_{i\geq1} a_i p^{n-i} - \sum_{i\geq1} p^{-i}(p-1 - a_i),$$
$$k=\sum_{i= 1}^{n}  a_i p^{n-i} +  \sum_{j\geq 1}  a_{j+n} p^{-j} - \sum_{i\geq1} p^{-i}(p-1 - a_i) .$$ 

Now looking at the coefficients of the non-negative powers of $p$, we obtain that $a_{i}= k_i$ for $1\leq i \leq n$. Also, from the negative powers of $p$ that need to have coefficient zero, we obtain that $a_{n+i} = (p-1)-a_i $ for all $i \geq 1$, and therefore $a_{2n+i} = (p-1)-((p-1)-a_i)=a_i $. This completes the proof. 
(Notice that all the series involved in the proof are absolutely convergent so we can rearrange them as we wish).
\end{proof} 


\begin{example}
For $p=2$, $n=3$,  the fixed points in $\FP(g_2^3)$ are the following

$x_0=\dfrac{0}{7}=0.\overline{000}_p $, \quad
$x_1=\dfrac{2}{9}=0.\overline{001110}_p $, \quad
$x_2=\dfrac{2}{7}=0.\overline{010}_p $, \quad
$x_3=\dfrac{4}{9}=0.\overline{011100}_p $, 

$x_4=\dfrac{4}{7}=0.\overline{100}_p $, \quad
$x_5=\dfrac{6}{9}=0.\overline{101010}_p $, \quad
$x_6=\dfrac{6}{7}=0.\overline{110}_p $, \quad
$x_7=\dfrac{8}{9}=0.\overline{111000}_p $. 
\end{example}

The following proposition shows how the function $g_p$ looks like when expressed in base $p$. 

\begin{proposition}\label{ginbasep} If $x=\sum_{i\ge 1}  \frac{a_i}{p^i}$ for $0\le a_i \le p-1,$ then
\[g_p(x)=
\begin{cases} 
\sum_{i\ge 1} \frac{a_{i+1}}{p^i}, &\text{if $a_1$  is even;}\\
\sum_{i\ge 1}\frac{p-1-a_{i+1}}{p^i}, &\text{if $a_1$  is odd.}
\end{cases} \]
\end{proposition}

\begin{proof} 
Notice that $\frac{a_1}{{p}}\leq x\leq \frac{a_1+1}{{p}}$. If  $a_1$  is even, then  $g(x)=-a_1+px=\sum_{i\ge 1} \frac{a_{i+1}}{p^i}$. If $a_1$  is odd, then $$g(x)=a_1+1-px=1-\left(\sum_{i\ge 1}\frac{a_{i+1}}{p^i}\right)=\sum_{i\ge 1}\frac{p-1-a_{i+1}}{p^i}.$$
\end{proof}

\begin{example}
The function $g_2$ permutes the fixed points in $\FP(g_2^3)$ as follows:
\begin{center}
$0.\overline{000}_2 \rightarrow 0.\overline{000}_2$

$0.\overline{001110}_2 \rightarrow 0.\overline{011100}_2 \rightarrow 0.\overline{111000}_2\rightarrow 0.\overline{001110}_2$ 

$0.\overline{010}_2 \rightarrow 0.\overline{100}_2 \rightarrow 0.\overline{110}_2\rightarrow 0.\overline{010}_2$ 

$0.\overline{101010}_2 \rightarrow 0.\overline{101010}_2$
\end{center}
\end{example}



\section{Bijection $B_\alpha$ }\label{sbij}


The goal of this section is to prove Theorem \ref{t2}.
Before, we need to prove the following propositions that are alternative representations of  $\pi_{p^n}(k)$  using the base $p$ expression of $k$. 

\begin{proposition}\label{pibasep}
If $k=(a_1a_2\ldots a_n)_p$ (where $k<p^n$), then $\pi_{p^n}(k)=(b_1 b_2\ldots b_n)_p$ where $b_1=a_1$, and $b_i=a_i$ when $b_1+\cdots +b_{i-1}$ is even or $b_i=p-1-a_i$ when $b_1+\cdots +b_{i-1}$ is odd.
\end{proposition}

\begin{proof}
Notice that  $a_1 p^{n-1}\le k<(a_1+1)p^{n-1}$. Also, it holds that $\pi_{p^n}(k)=(b_1 b_2\ldots b_n)_p$  has an expression in base $p$ with $n$ digits, where $b_1=k=a_1$ that accounts for the term $kp^{n-1}$ adding up, and all other $b_i$ depend on $\pi_{p^{n-1}}(k-a_1p^{n-1})$ in case $a_1$ is even or on $\pi_{p^{n-1}}(p^{n-1}-(k-a_1p^{n-1})-1)$ in case $a_1$ is odd.

Then if $a_1$ is even, by a similar argument, it holds that $\pi_{p^{n-1}}(k-a_1p^{n-1})$ in base $p$ has the same first digit as $k-a_1p^{n-1}=(a_2\ldots a_n)_p,$ namely $a_2$. Then $b_2=a_2$ and in general it holds that  $\pi_{p^{n-1}}(k-a_1p^{n-1})=(b_2\ldots b_n)_p$.   
By an inductive argument on $n$, the number of digits of $k$ in base $p$, we can assume that the result holds for $k-a_1p^{n-1}$ and then  $b_i=a_i$ when $b_2+\cdots +b_{i-1}$ is even or $b_i=p-1-a_i$ when $b_2+\cdots +b_{i-1}$ is odd. Since $b_1=a_1$ is even, the result holds as well for $k$.

On the other hand, notice that if  $p^{n-1}-1-(i-a_1p^{n-1})=(c_2\ldots c_n)_p$ every digit $c_i$ is the complement of $a_i$, meaning by this that $c_i=p-1-a_i$. This is true since all digits of $p^{n-1}-1$ in base $p$ are $p-1$.
 
Therefore if $a_1$ is odd, then $b_2=c_2=p-1-a_2$. Again by an inductive analysis we assume that  $b_i=c_i$ when $b_2+\cdots +b_{i-1}$ is even or $b_i=p-1-c_i$ when $b_2+\cdots +b_{i-1}$ is odd. Since $b_1=a_1$ is odd, then $b_i=c_i=p-1-a_i$ when $b_1+b_2+\cdots +b_{i-1}$ is odd or $b_i=p-1-c_i=a_i$ when $b_1+b_2+\cdots +b_{i-1}$ is even.
\end{proof}

 Notice that for $p$ odd, taking $p-1$ complement doesn't change the parity of the digits, and therefore we could use the parity of $a_1+\cdots +a_{i-1}$ to create the two cases in the previous proposition. In case $p=2$ things are a little bit different.

\begin{proposition}\label{pibasepeven}
If $p=2$ and $i=(a_1a_2\ldots a_n)_2$, then $\pi_{p^n}(i)=(b_1 b_2\ldots b_n)_2$ where $b_1=a_1$, and $b_i=0$ if $a_{i-1}=a_i$ or $b_i=1$ otherwise.
\end{proposition}


\begin{proof}
 We prove it inductively. From Proposition \ref{pibasep}, we know that $a_1=b_1$. We assume that $b_i=0$ if $a_{i-1}=a_i$ or $b_i=1$ otherwise, for values of $i$ smaller than $j$ and check that the result holds for $i=j$. We divide in two cases.
\begin{itemize}
    \item If $b_1+\ldots +b_{j-1}$ is even, then by induction hypothesis, there is an even number of changes on the $a_i$, and therefore $a_{j-1}=0$. Also, from Proposition \ref{pibasep}, we see that $b_j=a_j$. It is clear that if $a_j=0$, then $a_{j-1}=a_j$ and $b_j=0$, while if $a_j=1$, then $a_{j-1}\neq a_j$ and $b_j=1$.

    \item If $b_1+\ldots +b_{i-1}$ is odd, then we have an odd number of changes and $a_{j-1}=1$. Also, from Proposition \ref{pibasep}, we see that $b_j=1-a_j$. If $a_j=0$, $b_j=1$ and $a{j-1}\neq a_j$ while if  $a_j=1$, then $b_j=0$ and $a_{j-1}= a_j$.
\end{itemize}
In all cases we check that the result holds. By induction, it is true for all $i\le n$.
\end{proof}

\begin{proof}[Proof of Theorem \ref{t2}]
Consider a fixed point $x_k$ for $k=(a_1\ldots a_n)_p$.  From Proposition \ref{gfixedpts} and  \ref{periodicfp} we can express  $x_k=(0.\overline{a_1\ldots a_na'_1\ldots a'_n})_p$ where $a'_i=a_i$ in case that $(a_1\ldots a_n)_p$ is even, or $a'_i=p-1-a_i$ if $(a_1\ldots a_n)_p$ is odd.

Denote $\pi_{p^n}(k)=(b_1 b_2\ldots b_n)_p,$. Then by Proposition  \ref{pibasep} we have that
$b_1=a_1$, and $b_i=a_i$ when $b_1+\cdots +b_{i-1}$ is even or $b_i=p-1-a_i$ when $b_1+\cdots +b_{i-1}$ is odd. Then $(B_\alpha(x_k))^p=\alpha^{p(b_1\ldots b_n)_p}=\alpha^{(b_1\ldots b_n0)_p}=\alpha^{(b_2\ldots b_nb_1)_p}$, since $\alpha^{p^n}=\alpha$ in the finite field $\F_{p^n}$. We assume $k\neq 0$, since the case $k=0$ can be easily checked.

Now we consider two cases, depending if $a_1$ is even or odd.

\begin{itemize}
\item When $a_1$ is even:
  Then $g_p(x_i)=(0.\overline{a_2\ldots a_na'_1\ldots a'_na_1})_p=x_{(a_2\ldots a_na'_1)_p}$, due to Proposition \ref{ginbasep}. We denote $B_\alpha(g_p(x_k))=\alpha^{(d_2 \ldots d_nd_1)_p}$, where $(d_2 \ldots d_nd_1)_p=\pi_{p^n}((a_2\ldots a_na'_1)_p)$, 
  with $d_2=a_2$, $d_i=a_i$ if $d_2+\cdots +d_{i-1}$ is even or $d_i=p-1-a_i=c_i$ if $d_2+\cdots +d_{i-1}$ is odd, and $d_1=a'_1$ if $d_2+\cdots +d_{n}$ is even, or $d_1=p-1-a'_1=c'_1$ if $d_2+\cdots +d_{n}$ is odd.

Now we need to check that $d_i=b_i$ for $i=1, \ldots,n$.
 Inductively we can check that $b_i=d_i$ for $i\le 2$. Since $a_1$ is even, $b_2=a_2$.  For $i>2$, the definition of $d_i$ and $b_i$ are similar, with  cases depending on the parity of $d_2+\cdots +d_{i-1}$ and of $b_1+\cdots +b_{i-1}$. Since $b_1=a_1$ is even, by the induction hypothesis, both will have the same parity, and therefore $b_i=d_i$. We still need to check that $b_1=d_1$, we further consider two cases:

\begin{itemize}
\item If $p$ is odd, we know that $b_i$ and $a_i$ have the same parity. Consider two cases: If $d_2+\cdots +d_{n}=b_2+\cdots +b_{n}$ is even, then $d_1=a'_1$. Also $(a_1\ldots a_n)_p$ is even since $b_i$ and $a_i$ have the same parity, and  therefore $d_1=a'_1=a_1=b_1$. 

If $d_2+\cdots +d_{n}=b_2+\cdots +b_{n}$ is odd, then $d_1=p-1-a'_1$. Also $(a_1\ldots a_n)_p$ is odd in this case, and then $a'_i=p-1-a_i$, and therefore $d_1=a_1=b_1$.

\item  If $p=2$,  $(a_1\ldots a_n)_p$ has the same parity as $a_n$. Clearly $a_1=0$ since it is even. If $a_n=0$, then $d_1=a_1=0=b_1$. If $a_n=1$, then $a'_1=1$, and by Proposition \ref{pibasepeven}, $d_1=0$. Since $b_1=a_1=0$. Otherwise, if $a_n=1$, then $a'_1=1-a_1=1.$ Again in this case we conclude that $d_1=0=b_1$.

\end{itemize}

\item When $a_1$ is odd:
 In this case $g_p(x_i)=(0.\overline{c_2\ldots c_nc'_1\ldots c'_nc_1})_p=x_{(c_2\ldots c_nc'_1)_p}$, where $c_i=p-1-a_i$ and $c'_i=p-1-a'_i,$  due to Proposition \ref{ginbasep}. We denote $B_\alpha(g_p(x_k))=\alpha^{(d_2 \ldots d_nd_1)_p}$, where $(d_2 \ldots d_nd_1)_p=\pi_{p^n}((c_2\ldots c_nc'_1)_p)$. 

Again we  check  inductively that $d_i=b_i$ for $i=1, \ldots,n$.
Since $a_1$ is odd,  $b_2=p-1-a_2=c_2$. Also $d_2=c_2$, therefore $b_2=d_2$.  For $i>2$, the definition of $d_i$ and $b_i$ have cases depending on the parity of $d_2+\cdots +d_{i-1}$ and of $b_1+\cdots +b_{i-1}$. Since $b_1=a_1$ is odd, by the induction hypothesis, both will have different parity. 
If $d_2+\cdots +d_{i-1}$ is even, then $d_i=c_i$, while $b_1+\cdots +b_{i-1}$ is odd and then $b_i=p-1-a_i=c_i$. Similarly, if $d_2+\cdots +d_{i-1}$ is odd, then $d_i=p-1-c_i=a_i$, while $b_1+\cdots +b_{i-1}$ is even and then $b_i=a_i$. 
Therefore $b_i=d_i$. We still need to check that $b_1=d_1$.
 
\begin{itemize}
\item If $p$ is odd, we know that $b_i$ and $a_i$ have the same parity. Consider two cases: If $d_2+\cdots +d_{n}=b_2+\cdots +b_{n}$ is even, then $d_1=c'_1$. Now $(a_1\ldots a_n)_p$ is odd since $b_i$ and $a_i$ have the same parity for $i\le 2$  and here he assume $a_1$ and $p$ to be odd.  Therefore $b_1=p-1-a'_1=c'_1=d_1$. 

If $d_2+\cdots +d_{n}=b_2+\cdots +b_{n}$ is odd, then $d_1=p-1-a'_1$, and $(a_1\ldots a_n)_p$ is even now. Then $a'_i=p-1-a_i$, and therefore $d_1=a_1=b_1$.

\item  If $p=2$,  $(a_1\ldots a_n)_p$ has the same parity as $a_n$. This time $a_1=1$ since it is odd. If $a_n=0$,  $d_1=a_1=1=b_1$. If $a_n=1$, then $a'_1=0$, and by Proposition \ref{pibasepeven}, $d_1=1$, since $a_n\neq a'_1$. It holds then that $b_1=a_1=1=d_1$. Otherwise, if $a_n=1$, then $a'_1=1-a_1=0.$ Again in this case we conclude that $d_1=1=b_1$.
\end{itemize}
\end{itemize}
\end{proof}

\begin{example}
  Using \emph{python} and \emph{SageMath} math \cite{sage} we were are able  to compute  our bijection. We present here the case $p=2$, $n=4$.  For this, by default, \emph{sage}  selects as irreducible polynomial of order $n$ the Conway polynomials,  those polynomials are used here as well. For more information about Conway polynomials and other possible choices for irreducible polynomials in finite field algorithmic constructions, you can see \cite{conway}. 
Since the Conway polynomial for a finite field is chosen so as to be compatible with the Conway polynomials of each of its sub-fields \cite{conway}, our bijection using those is compatible within the corresponding sub-fields. 

  We also plot the orbits of the fixed points of $g_p^n$ by the action of $g_p$ for $p=2,\,n=4$ and for $p=3,\,n=3$.

\begin{tabular}{ |p{1cm}|p{3.5cm}|p{3.5cm}|p{1cm}|p{2.8cm}|  }
 \hline
 \multicolumn{5}{|c|}{Bijection $B_{\alpha}$ for $p=2$, $n=4$, and $\alpha^4+\alpha+1$ in $\F_{2^4}$}\\
 \hline
 $i$ & $x_i$ & $g_p(x_i)$  & $\pi_{p^n}(i)$ & $B_{\alpha}(x_i)$ \\
 \hline
0 & 0.0 & 0.0 & 0 & 0 \\
1 & 0.11764705882352941 & 0.23529411764705882 & 1 & $\alpha$ \\
2 & 0.13333333333333333 & 0.26666666666666666 & 3 & $\alpha^3$ \\
3 & 0.23529411764705882 & 0.47058823529411764 & 2 & $\alpha^2$ \\
4 & 0.26666666666666666 & 0.5333333333333333 & 6 & $\alpha^3 + \alpha^2$ \\
5 & 0.35294117647058826 & 0.7058823529411765 & 7 & $\alpha^3 + \alpha + 1$ \\
6 & 0.4 & 0.8 & 5 & $\alpha^2 + \alpha$ \\
7 & 0.47058823529411764 & 0.9411764705882353 & 4 & $\alpha + 1$ \\
8 & 0.5333333333333333 & 0.9333333333333333 & 12 & $\alpha^3 + \alpha^2 + \alpha + 1$ \\
9 & 0.5882352941176471 & 0.8235294117647058 & 13 & $\alpha^3 + \alpha^2 + 1$ \\
10 & 0.6666666666666666 & 0.6666666666666667 & 15 & 1 \\
11 & 0.7058823529411765 & 0.588235294117647 & 14 & $\alpha^3 + 1$ \\
12 & 0.8 & 0.3999999999999999 & 10 & $\alpha^2 + \alpha + 1$ \\
13 & 0.8235294117647058 & 0.3529411764705883 & 11 & $\alpha^3 + \alpha^2 + \alpha$ \\
14 & 0.9333333333333333 & 0.1333333333333333 & 9 & $\alpha^3 + \alpha$ \\
15 & 0.9411764705882353 & 0.11764705882352944 & 8 & $\alpha^2 + 1$ \\

 \hline
\end{tabular}

\begin{figure}[h]
    \centering
    \includegraphics[width=10.5cm]{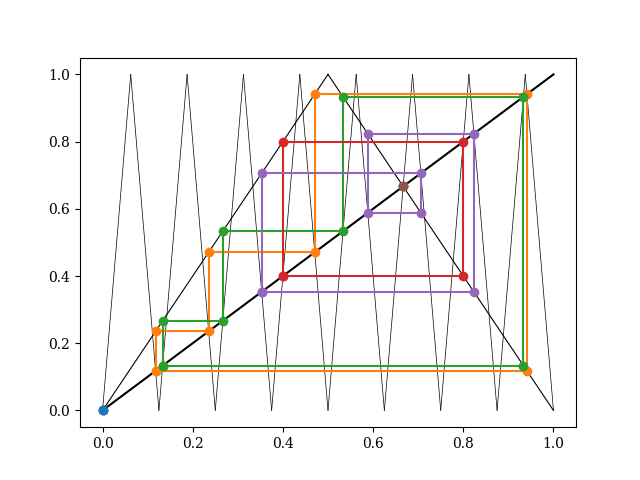}
    \label{gpn}
    \caption{Fixed points of $g_2^4$ and their orbits under $g_2$. }
\end{figure}


\begin{figure}[h]
    \centering
    \includegraphics[width=13.5cm]{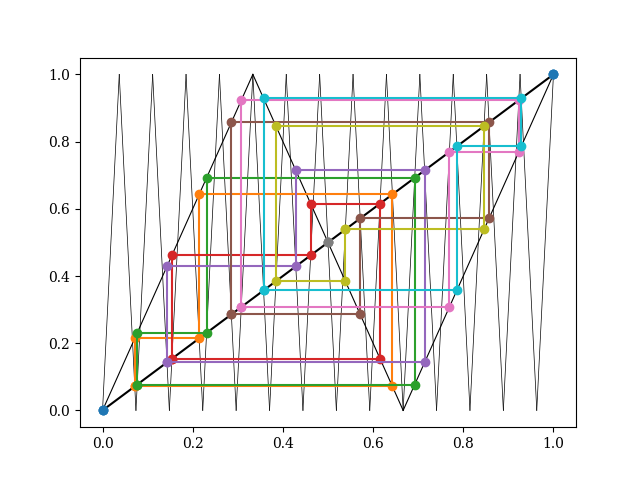}
    \label{gpn}
    \caption{Fixed points of $g_3^3$ and their orbits under $g_3$. }
\end{figure}

\end{example}

The bijection $B_\alpha$ also relates $\FP(g_p^n)$ with the multiplication in $\F_{p^n}$, for any given $\alpha$.

\begin{proposition}\label{t3} The bijection $B_\alpha:\FP(g_p^n)\rightarrow \F_{p^n}$ satisfy that $B_\alpha(x_i)B_\alpha(x_j)=B_\alpha(x_r)$ with 
$$r=\pi_{p^n}^{-1}((\pi_{p^n}(i)+\pi_{p^n}(j))\mod p^n-1),$$ for any $0< i,\,j < p^n,$ where the class representative modulo $p^n-1$ must be taken from $1$ to $p^n-1$.
\end{proposition}
\begin{proof}
  Since $B_\alpha(x_i)=\alpha^{\pi_{p^n}(i)}$ and $B_\alpha(x_j)=\alpha^{\pi_{p^n}(j)},$ therefore $B_\alpha(x_i)B_\alpha(x_j)=\alpha^{\pi_{p^n}(i)+\pi_{p^n}(j)}$. Since $\pi_{p^n}(r)=\pi_{p^n}(i)+\pi_{p^n}(j)$, then $B_\alpha(x_i)B_\alpha(x_j)=B_\alpha(x_r).$ 
\end{proof}
It would be good to connect our bijection with the sum structure of the finite field. We have some freedom for the generator $\alpha$ for that purpose, see \cite{conway} for other possible choices of a minimal polynomial for a generator.

\section{Bijection for other continuous functions}\label{scheb}

\begin{definition}
Let $u<v \in \R$. We say that such a function $f_p:[u,v]\rightarrow[u,v]$ \emph{goes up and down $p$ times} if it is a continuous function so that there are numbers $u=z_0<z_1<z_2<\cdots<z_p=v$ where $f_p(z_k)$ is always $u$ or $v$ for any $k$, and $f_p$ is strictly increasing or strictly decreasing when restricted to any interval $[z_k, z_{k+1}]$ for $k<p$. We will also require that the function $f_p$ has exactly $p$ fixed points, one on each interval $[z_k, z_{k+1}]$.
\end{definition}

Let $p$ be a prime number. We want to extend Theorem \ref{t2} to other functions $f_p:[u,v]\rightarrow [u,v]$ going up and down $p$ times. We restrict to the case that there is a continuous bijection $h:[0,1]\rightarrow [u,v]$ that is a homeomorphism between both functions $g_p$ and $f_p$ (so that  $f_p\circ h =h \circ g_p$). In this case we can easily extend our results to $f_p$ as well. For us this is enough to include another important family of continuous functions going up and down, namely the \emph{Chebyshev Polynomials}.

\begin{theorem}\label{t2h}
If $f_p$ \emph{goes up and down $p$ times} and there is a continuous bijection $h:[0,1]\rightarrow [u,v]$ so that  $f_p\circ h =h \circ g_p$, then there is a bijection $B_f:\FP(f_p^n)\rightarrow \F_{p^n}$ so that $(B_f(y_i))^p=B_f(f_p(y_i))$ for any fixed point $y_i\in \FP(f_p^n).$ 
\end{theorem}

\begin{proof}
  A point $x\in [0,1]$ is a fixed point of $g_p^n$ if and only if $h(g_p^n(x))=h(x)$  since $h$ is bijective, and  $f_p^n(h(x))=h(g_p^n(x))=h(x)$ if and only if $h(x)$ is also a fixed point of $f_p^n$.
  Then, if $0=x_0<x_1<\ldots <x_{p^n-1}$ are the $p-1$ fixed points of the function $g_p^n$  and $y_i=h(x_i)$ for $0\le i<p^n$, then $y_0,\, y_1, \ldots ,\,y_{p^n-1}$ are the $p^n$ fixed points of the function $f_p^n$. Notice that  $h$ preserves order or completely inverts it. 

  We can simply take $B_f=B_\alpha\circ h^{-1}$, so that if $B_f(y_i)=B_\alpha(x_i)$ for all $0\le i<p^n.$
Notice that if $g_p(x_i)=x_j$, then $f_p(y_i)=y_j$, since $f_p(y_i)=f_p(h(x_i))=h(g_p(x_i))=h(x_j)=y_j$. Then $$(B_f(y_i))^p=(B_\alpha(x_i))^p=B_\alpha(g_p(x_i))=B_\alpha(x_j)=B_f(y_j)=B_f(f_p(y_i)).$$ \end{proof}

\subsection{Example: Chebyshev Polynomials}

As an example we will see how to extend our results to \emph{Chebyshev Polynomials.}
These are widely known polynomials, very important  in numerical analysis \cite{Chebyshev}.

\begin{definition}[Chebyshev polynomials]
 The family of Chebyshev polynomials $T_k \in \R[x]$  can be defined recursively by $T_0(x)=1$, $T_1(x)=x$, and $T_{k+1}(x)=2x T_k(x)-T_{k-1}(x)$.
\end{definition}

\begin{example}
  The next polynomials in the sequence are $T_2(x)=2x^2-1$, $T_3(x)=4x^3-3x$, $T_4(x)=8x^4-8x^2+1$, $T_5(x)=16x^5-20x^3+5x$, and $T_6(x)=32x^6-48x^4+18x^2-1$.
\end{example}
\begin{figure}[h]
    \centering
    \includegraphics[width=10.5cm]{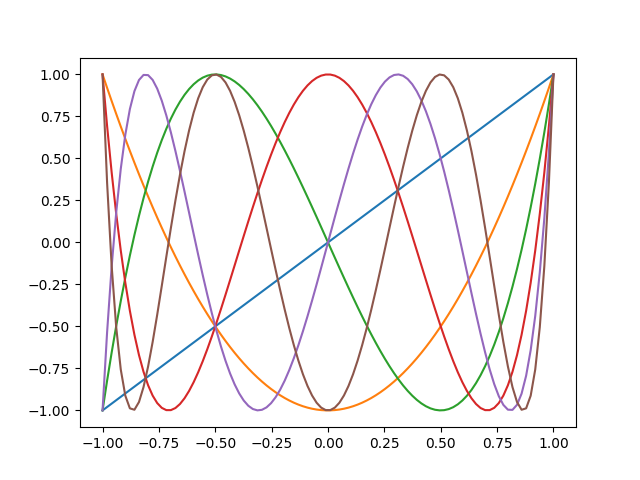}
    \label{gpn}
    \caption{Chebyshev Polynomials $T_1,\ldots,T_6$ on the interval $[-1,1]$.}
\end{figure}

\begin{proposition}\label{chebcos}
  The Chebyshev polynomial $T_n$ is the expression for $\cos(n\theta)$ in terms of $\cos(\theta)$, namely it holds that $$\cos(n\theta)=T_n(cos(\theta)).$$
\end{proposition}

\begin{corollary}\label{chebcorol}
  If $m$, $n$ are positive integers, then $T_m\circ T_n=T_{mn}$. Therefore also $T_p^n=T_{p^n}$.
\end{corollary}

Proposition \ref{chebcos} and Corollary \ref{chebcorol} can be found in \cite{Chebyshev}. In fact Proposition \ref{chebcos} is often used as a definition.
The polynomials $T_p$ restricted to the interval $[-1,1]$ are functions going up and down $p$ times. We want to extend our bijection to fixed points of Chebyshev polynomials as well. This is an easy application of Theorem \ref{t2h}.

\begin{proposition}\label{chebhg}
   The continuous bijection $h:[0,1]\rightarrow [-1,1]$ is given by $h(x)=\cos(\pi x)$. It holds that $T_p\circ h =h \circ g_p$ for any value of $p$.
\end{proposition}
\begin{proof}
  We need to verify that $T_p( \cos(\pi x)) =\cos(\pi g_p(x))$. Take $x\in [\frac kp, \frac{k+1}{p}]$. We recall Definition \ref{defg} and split in two cases.
  
   If $k$ is even, then $g_p(x)=px-k$ and then $\cos(\pi g_p(x))=\cos(p\pi x-k\pi)=\cos(p\pi x)$, since $\cos$ is $2\pi$-periodic and $k$ is even. But by Proposition \ref{chebcos} we have that $T_p( \cos(\pi x)) =\cos(p\pi x)=\cos(\pi g_p(x))$.

  If $k$ is odd, then $g_p(x)=k-px$ and then $\cos(\pi g_p(x))=\cos(k\pi-p\pi x)=\cos(\pi-p\pi x)=\cos(p\pi x)$, since $\cos$ is $2\pi$-periodic and $\cos(x)=\cos(\pi -x)$. Again by Proposition \ref{chebcos} we have that $T_p( \cos(\pi x)) =\cos(p\pi x)=\cos(\pi g_p(x))$. We conclude that $T_p\circ h =h \circ g_p.$ 
\end{proof}

\begin{corollary}
 There is a bijection $B_f:\FP(T_p^n)\rightarrow \F_{p^n}$ so that $(B_f(y_i))^p=B_f(T_p(y_i))$ for any fixed point $y_i\in \FP(T_p^n).$ 
\end{corollary}

\begin{proof}
  This is a consequence of Proposition \ref{chebhg} and Theorem \ref{t2h}.
\end{proof}

This bijection allows us to factor $T_p^n(x)-x$, since we know all fixed points via the bijection.

\begin{proposition}
  The Chebyshev polynomial $T_p^n$ can be expressed as 
  $$T_p^n(x)=x+\prod_{k=0}^{(p^n-1)/2}\left(x-\cos\left(\frac{2k\pi}{p^n-1}\right)\right)\left(x-\cos\left(\frac{(2k+1)\pi}{p^n+1}\right)\right).$$
  
\end{proposition}

\begin{proof}
  The fixed points of $T_p^n(x)-x$ are of the form $y_i=h(x_i)=\cos(\pi x_i)$, where $x_0<x_1<\ldots <x_{p^n-1}$  are the  fixed points of $g_p^n$, as it can be observed in the proof of Theorem \ref{t2h}. Therefore
$$T_p^n(x)-x=\prod_{i=0}^{p^n-1}(x-\cos(x_i)).$$
 
  By Proposition \ref{gfixedpts}, and arranging the roots in pairs we obtain the desired expression.
\end{proof}

For other functions $f_p$ that go up and down $p$ times, we would like to know if there is a continuous bijection $h:[0,1]\rightarrow [u,v]$ so that $f_p\circ h =h \circ g_p$. It is possible to find the values of the function $h(x)$ for rational values $x$ with denominator of the form $p^k$, recursively on $k$. In these values end up to be increasing, and if such function $h$ exists, for all $x\in [0,1]$,  $h(x)$ must be the limit of the rational values approaching $x$. For us it is not clear that this limit always exist, or what extra conditions should be requested for $f_p$ to accomplish that.

\section{ Functions going up or down}\label{sgeneralud}

Another  direction in which we can generalize our bijection is  to look at functions that are not necessarily continuous, that go up or down $p$ times, but following different patterns.
\begin{definition}\label{defgI}
  Let $p$ be a prime number, and $I\subseteq \{0,1,2,\ldots,p-1\}$.
  The function  $g_{p,I}:[0,1]\rightarrow [0,1]$ given by
\[g_{p,I}(x)=
\begin{cases} 
{p}x-k, &\text{for $\frac{k}{{p}}\leq x< \frac{k+1}{{p}}$ with $k\in I$.}\\
k+1-{p}x, &\text{for $\frac{k}{{p}}\leq x< \frac{k+1}{{p}}$ with $k\notin I$.}
\end{cases} \]
\end{definition}
In this case $g_{p,I}$ is a piece-wise linear function where $I$ denotes the set of indices $k$ where $g_{p,I}$ is increasing. Notice that if $I$ is the set of even numbers, then $g_{p,I}=g_p$. Also, if  $I= \{0,1,2,\ldots,p-1\}$, then $g_{p,I}$ coincides with the fractional part of $px$.

The function $g_{p,I}^n$ is linear with slope $\pm {p^n}$ on each open interval $(\frac{k}{{p^n}}, \frac{k+1}{{p^n}})$.
\begin{definition}
  The set $I_{p^n}\subseteq \{0,1,2,\ldots,p^n-1\}$ is the set of indices where $g_{p,I}^n$ is increasing on the interval $(\frac{k}{{p^n}}, \frac{k+1}{{p^n}})$.
\end{definition}

Then we have $I=I_{p^1}$.
Using this definition we can find all fixed points of $g_{p,I}^n$ similar to what we did in  Proposition \ref{gfixedpts}.

\begin{proposition}\label{gIfixedpts}
There are exactly $p^n$ fixed points of $g_{p,I}^n$.
Let $x_0<x_1<\ldots <x_{p^n-1}$ be the fixed points of the function $g_{p,I}^n$.
Then, for $0\le k<p^n$ we have that 

\[
x_k=
\begin{cases}
  \frac{k}{p^n-1}, &\text{if $k\in I_{p^n}$}
  \\
  \frac{k+1}{p^n+1}, &\text{ otherwise.}
\end{cases} 
\]
\end{proposition}

Also, analogue to \ref{ginbasep}, we get the following proposition that helps us to see how $g_{p,I}$ looks like in base $p$. 
\begin{proposition}\label{ginbasepI} If $x=\sum_{i\ge 1}  \frac{a_i}{p^i}$ for $0\le a_i \le p-1,$ then
\[g_{p,I}(x)=
\begin{cases} 
\sum_{i\ge 1} \frac{a_{i+1}}{p^i}, &\text{if $a_1\in I$  }\\
\sum_{i\ge 1}\frac{p-1-a_{i+1}}{p^i}, &\text{otherwise.}
\end{cases} \]
\end{proposition}

The proofs of the previous propositions are just the same as before.

We can prove analogous results to Theorem \ref{t1} and \ref{t2}, but this time we need to construct a new bijection $B_{\alpha,I}$  using a different permutation $\pi_{p^n,I}$, analogous to Definition \ref{pipn}.
\begin{definition}\label{pipnI}
The permutation  $\pi_{p^n,I}$ is defined recursively 
as follows: take $\pi_{p^1,I}$ to be the identity map from 0 to $p-1$. Then to define $\pi_{p^n,I}$, if $a p^{n-1}\le k<(a+1)p^{n-1}$ take
 \[\pi_{p^n,I}(k)=\begin{cases}
   ap^{n-1} + \pi_{p^{n-1},I}(k-ap^{n-1}) \text{ for $a\in I$}\\
   ap^{n-1} + \pi_{p^{n-1},I}(p^{n-1}-(k-ap^{n-1})-1)  \text{ for $a\notin I$}
\end{cases}\]
\end{definition}

\begin{example}
For $p=2$ and $p=3$, and $I=\varnothing$, this is  how $\pi_{p^n,I}$  permutes the numbers $0, 1, \dots ,p^n-1$ :

$\pi_{2^1,I}: 0, 1$

$\pi_{2^2,I}: 1, 0, 3, 2$

$\pi_{2^3,I}: 2, 3, 0, 1, 6, 7, 4, 5$

$\pi_{2^4,I}: 5,4,7,6,1,0,3,2, 13, 12, 15, 14, 9,8, 11, 10$

$\pi_{3^1,I}: 0, 1, 2$

$\pi_{3^2,I}: 2, 1, 0, 5, 4, 3, 8, 7, 6$

$\pi_{3^3,I}: 6, 7, 8, 3, 4, 5, 0, 1, 2,  15,  16, 17, 12, 13, 14, 9, 10, 11, 24, 25, 26, 21, 22, 23, 18, 19, 20$
\end{example}

The following propositions would be needed to prove Theorem \ref{t2I}.

\begin{proposition}\label{propI}  
 If $a_1 \in I$, then  $(a_1\ldots a_{n})_p \in I_{p^{n}}$ if and only if $(a_2\ldots a_{n})_p\in I_{p^{n-1}}$. On the other hand, if $a_1 \notin I$, then  $(a_1\ldots a_{n})_p \in I_{p^{n}}$ if and only if $(c_2\ldots c_{n})_p\notin I_{p^{n-1}}$where $c_i=p-1-a_i$ for $i=2, \ldots, n$. 
\end{proposition}  

\begin{proof}
   Notice that  $k=(a_1\ldots a_{n})_p \in I_{p^{n}}$ if and only if the function $g_{p,I}^n$ in the interval  $(\frac{k}{p^n},\frac{k+1}{p^n})$ is increasing. In case $a_1 \in I$, then $$g_{p,I}^n(x)= g_{p,I}^{n-1}(g_{p,I}(x))=g_{p,I}^{n-1}(px-a_1)$$ for $x$ in that interval.  In that case $px-a_1\in (\frac{(a_2\ldots a_{n})_p}{p^n},\frac{(a_2\ldots a_{n})_p+1}{p^n})$ is increasing, and then $k=(a_1\ldots a_{n})_p \in I_{p^{n}}$ if and only if $k=(a_2\ldots a_{n})_p \in I_{p^{n-1}}.$
  Now if  $a_1\notin I$, $$g_{p,I}^n(x)= g_{p,I}^{n-1}(g_{p,I}(x))=g_{p,I}^{n-1}(a_1+1-px)$$ for $x\in (\frac{k}{p^n},\frac{k+1}{p^n})$. In this case $a_1+1-px\in \frac{(c_2\ldots c_{n})_p}{p^n},\frac{(c_2\ldots c_{n})_p+1}{p^n})$ is decreasing. Therefore  $(a_1\ldots a_{n})_p \in I_{p^{n}}$ if and only if $(c_2\ldots c_{n})_p\notin I_{p^{n-1}}$.
\end{proof}

\begin{proposition}\label{prop-pipnI}
Let $n\ge 2$. If $ k=pb+d $ for $b$, $d$ integers with $0\le d <p$ then
\[\pi_{p^n,I}(k)=p\pi_{p^{n-1},I}(b) +\begin{cases}
   d    \text{ for $b \in I_{p^{n-1}}$}\\
   (p-d-1)\text{ otherwise.}
\end{cases}\]
\end{proposition}

\begin{proof}

  By induction, for $n=2$, then both expressions coincide, since
  \[\pi_{p^2,I}(k)=\begin{cases}
  pb + d    \text{ for $b \in I_{p^{1}}=I$}\\
  pb + (p-1-d)\text{ otherwise.}
  \end{cases}\]
  is equivalent to the expression obtained for $\pi_{p^2,I}(k)$ from Definition \ref{pipnI}, where $a=b$ and $d=k-ap$.
  
  Now assume Proposition \ref{prop-pipnI} for a fixed $n\le 2$ and let us check it for $n+1$. We use the notation base $p$ here, so take $k=(a_1\ldots a_{n+1})_p<p^{n+1}$.
  By definition,
  \[\pi_{p^{n+1},I}(k)=\begin{cases}
   a_1p^{n} + \pi_{p^{n},I}(k-a_1p^{n}) \text{ for $a_1\in I$}\\
   a_1p^{n} + \pi_{p^{n-1},I}(p^{n-1}-(k-a_1p^{n-1})-1)  \text{ for $a_1 \notin I$}
\end{cases}\]

\begin{itemize}
\item  Consider first the case $a_1\in I$. Here $k-a_1p^{n}=(a_2\ldots a_{n+1})_p$ and by hypothesis, we have that
  \[\pi_{p^{n},I}(k-a_1p^{n})=\begin{cases}
  p\pi_{p^{n-1},I}((a_2\ldots a_{n})_p) + a_{n+1}    \text{ for $b' \in I_{p^{n-1}}$}\\
  p\pi_{p^{n-1},I}((a_2\ldots a_{n})_p) + (p-a_{n+1}-1)\text{ otherwise.}
  \end{cases}\]

  for $b'=(a_2\ldots a_{n})_p$. Then,
  \[   \pi_{p^{n+1},I}(k) =   a_1p^{n} + p\pi_{p^{n-1},I}((a_2\ldots a_{n})_p) +
  \begin{cases}
   a_{n+1}    \text{ for $b' \in I_{p^{n-1}}$}\\
   (p-a_{n+1}-1)\text{ otherwise}
  \end{cases}
  \]
  
  \[  \pi_{p^{n+1},I}(k) =   p\pi_{p^{n},I}((a_1\ldots a_{n})_p) +
  \begin{cases}
   a_{n+1}    \text{ for $(a_2\ldots a_{n})_p \in I_{p^{n-1}}$}\\
   (p-a_{n+1}-1)\text{ otherwise}
    \end{cases}\]

  Using $a_1\in I$ and by Proposition \ref{propI}, we get the desired condition.

\item Now for the case $a_1\notin I$. Here $(p^{n-1}-(k-a_1p^{n-1})-1)=(c_2\ldots c_{n+1})_p$ where $c_i=p-a_i-1$. By hypothesis, we have that

  \[\pi_{p^{n},I}((c_2\ldots c_{n+1})_p)=\begin{cases}
  p\pi_{p^{n-1},I}((c_2\ldots c_{n})_p) + c_{n+1}    \text{ for $b' \in I_{p^{n-1}}$}\\
  p\pi_{p^{n-1},I}((c_2\ldots c_{n})_p) + a_{n+1}\text{ otherwise.}
  \end{cases}\]

  for $b'=(c_2\ldots c_{n})_p$. Then,
\[
  \pi_{p^{n+1},I}(k) =   a_1p^{n} + p\pi_{p^{n-1},I}((c_2\ldots c_{n})_p) +
  \begin{cases}
   c_{n+1}    \text{ for $b' \in I_{p^{n-1}}$}\\
   a_{n+1}\text{ otherwise.}
  \end{cases}\]

  \[\pi_{p^{n+1},I}(k)  =   p\pi_{p^{n},I}((a_1\ldots a_{n})_p) +
  \begin{cases}
   (p-a_{n+1}-1)    \text{ for $(c_2\ldots c_{n})_p \in I_{p^{n-1}}$}\\
   a_{n+1}\text{ otherwise.}
   \end{cases}\]

Now by the second part of Proposition \ref{propI}, the result holds.
  
\end{itemize}

\end{proof}

\begin{definition}[Bijection $B_{\alpha,I}$]
 Let $\alpha$ be a primitive root (i.e., a generator of the multiplicative group) of $\F_{p^n}$ (\cite{ffields}) and  let $x_0<x_1<\cdots<x_{p^{n}-1}$ be the fixed points of $g_{p,I}^n$. 
 We define the bijection $B_{\alpha,I}:\FP(g_{p,I}^n)\rightarrow \F_{p^n}$ by  
 $B_\alpha(x_{\pi_{p^n,I}(0)})=0\in \F_{p^n}$ or  $B_\alpha(x_k)=\alpha^{\pi_{p^n,I}(k)}$ for all other $k>0$. 
\end{definition}

The following theorem generalizes Theorem \ref{t2}. We keep the initial proof of this Theorem \ref{t2}, since the continuous case where $I$ is the set of even numbers less than $p$ has some extra structure and provides a good example of the general theory, where the sets $I_{p^{n}}$ and functions $\pi_{p^{n},I}$ could be described in more detail.

\begin{theorem}\label{t2I}
The function $B_{\alpha,I}$ is a bijection and satisfy that $B_{\alpha,I}(x_k)^p=B_{\alpha,I}(g_{p,I}(x_k))$ for any fixed point $x_k\in \FP(f_p^n).$ 
\end{theorem}

\begin{proof}

Consider a fixed point $x_k$ for $k=(a_1\ldots a_n)_p$.  From Proposition \ref{gIfixedpts} and  \ref{periodicfp} we can express  $x_k=(0.\overline{a_1\ldots a_na'_1\ldots a'_n})_p$ where $a'_i=a_i$ in case that $(a_1\ldots a_n)_p\in I_{p^n}$ , or $a'_i=p-1-a_i$ if $(a_1\ldots a_n)_p \notin I_{p^n}$.

Denote $\pi_{p^n,I}(k)=(b_1 b_2\ldots b_n)_p,$. 
Then $(B_\alpha(x_k))^p=\alpha^{p(b_1\ldots b_n)_p}=\alpha^{(b_1\ldots b_n0)_p}=\alpha^{(b_2\ldots b_nb_1)_p}$, since $\alpha^{p^n}=\alpha$ in the finite field $\F_{p^n}$. 
Now we consider two cases, depending if $a_1$ is in $I$ or not.

\begin{itemize}
\item If $a_1 \in I$:
  Then $g_p(x_i)=(0.\overline{a_2\ldots a_na'_1\ldots a'_na_1})_p=x_{(a_2\ldots a_na'_1)_p}$, due to Proposition \ref{ginbasepI}. We denote $B_\alpha(g_p(x_k))=\alpha^{(d_2 \ldots d_nd_1)_p}$, where $(d_2 \ldots d_nd_1)_p=\pi_{p^n,I}((a_2\ldots a_na'_1)_p)$. Now we need to check that $d_i=b_i$ for $i=1, \ldots,n$.
By Proposition \ref{prop-pipnI}, then 
\[\pi_{p^n,I}((a_2\ldots a_na'_1)_p)=p\pi_{p^{n-1},I}((a_2\ldots a_n)_p) + \begin{cases}
   a'_1    \text{ for $(a_2\ldots a_n)_p \in I_{p^{n-1}}$}\\
  (p-1-a'_1)\text{ otherwise.}
\end{cases}\]

Therefore  $(d_2 \ldots d_n)_p=\pi_{p^{n-1},I}((a_2\ldots a_n)_p)$, while $d_1= a'_1$ if $(a_2\ldots a_n)_p \in I_{p^{n-1}}$ or $d_1= p-1-a'_1$ otherwise.
  
On the other hand, since $a_1\in I$, by Definition \ref{pipnI}, we have that
$(b_1 b_2\ldots b_n)_p=a_1p^{n-1} + \pi_{p^{n-1},I}((a_2\ldots a_n)_p)$ and therefore $(b_2 \ldots b_n)_p=\pi_{p^{n-1},I}((a_2\ldots a_n)_p)=(d_2 \ldots d_n)_p$, while $b_1=a_1$. We use here Proposition \ref{propI}, to see that $(a_1\ldots a_{n})_p \in I_{p^{n}}$ if and only if $(a_2\ldots a_{n})_p\in I_{p^{n-1}}$. Therefore if $(a_1\ldots a_{n})_p \in I_{p^{n}}$, then  
$b_1=a_1=a'_1=d_1$, and if $(a_1\ldots a_{n})_p \notin I_{p^{n}}$, is easy to check that $b_1=d_1$ as well.

\item If $a_1\notin I$:
 In this case $g_p(x_i)=(0.\overline{c_2\ldots c_nc'_1\ldots c'_nc_1})_p=x_{(c_2\ldots c_nc'_1)_p}$, where $c_i=p-1-a_i$ and $c'_i=p-1-a'_i,$ due to Proposition \ref{ginbasepI}. We denote $B_\alpha(g_p(x_k))=\alpha^{(d_2 \ldots d_nd_1)_p}$, where $(d_2 \ldots d_nd_1)_p=\pi_{p^n}((c_2\ldots c_nc'_1)_p)$. 

 Again we need to check   that $d_i=b_i$ for $i=1, \ldots,n$. This time, since $a_1\notin I$, we have that $(b_1 b_2\ldots b_n)_p=a_1p^{n-1} + \pi_{p^{n-1},I}((c_2\ldots c_n)_p)$, while by Proposition \ref{prop-pipnI} we have that
\[\pi_{p^n,I}((c_2\ldots c_nc'_1)_p)=p\pi_{p^{n-1},I}((c_2\ldots c_n)_p) + \begin{cases}
   c'_1    \text{ for $(c_2\ldots c_n)_p \in I_{p^{n-1}}$}\\
   p-1-c'_1\text{ otherwise.}
   \end{cases}
   \]

   Then $(b_2 \ldots b_n)_p=\pi_{p^{n-1},I}((c_2\ldots c_n)_p)=(d_2 \ldots d_n)_p$, while $b_1=a_1$ and $d_1=c'_1$ for $(a_2\ldots a_n)_p \in I_{p^{n-1}}$, or $d_1=a'_1$ otherwise. This time we use the second part of Proposition \ref{propI} to see that $(a_1\ldots a_{n})_p \in I_{p^{n}}$ if and only if $(c_2\ldots c_{n})_p\notin I_{p^{n-1}}$.
   We check both cases, first if $(a_1\ldots a_{n})_p \in I_{p^{n}}$, then $a_1=a'_1$, while $(c_2\ldots c_{n})_p\notin I_{p^{n-1}}$ implies $d=p-1-c'_1=a'_1=a_1=b_1$. 
   On the other hand, $(a_1\ldots a_{n})_p \in I_{p^{n}}$ implies that $a'_1=c_1=p-1-a_1$, while $(c_2\ldots c_{n})_p\in I_{p^{n-1}}$ and therefore $d_1=c'_1=p-1-a'_1=a_1=b_1$. In any case, $b_i=d_i$, as desired.
   
\end{itemize}
\end{proof}

\begin{example}
For $p=3$, $n=2$ and $p=3$, $n=3$ with $I=\{2\}$, we have that $\pi_{3^n,I}$ permutes the numbers from 0 to $3^n-1$ as follows:

$\pi_{3^2,I}:$ 2 1 0 5 4 3 6 7 8

$\pi_{3^3,I}:$ 8 7 6 3 4 5 0 1 2 17 16 15 12 13 14 9 10 11 20 19 18 23 22 21 24 25 26

We also compute $I_{3^2}=\{1,2,4,5,8\}$ and $I_{3^3}=\{1, 2, 5, 8, 10, 11, 14,17,19,20,22,23,26\}$ and plot the fixed points for $g_{3^2,I}$ and $g_{3^3,I}$, and its orbits under the action of $g_{3,I}$.


\begin{figure}[h!]
    \centering
    \includegraphics[width=10.5cm]{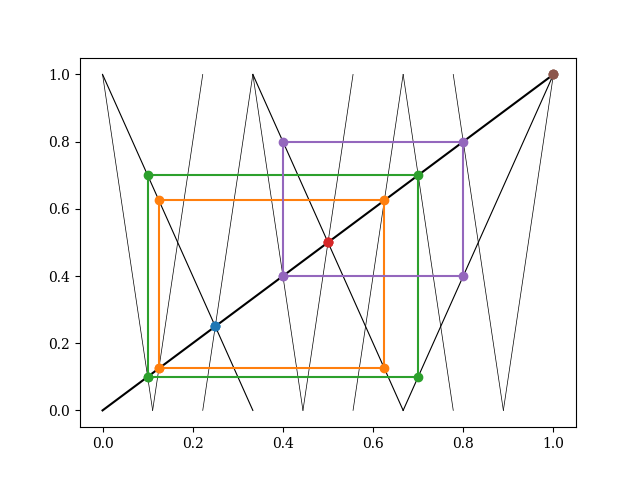}
    \label{gpn}
    \caption{Fixed points of $g_{3^2,I}$ for $I=\{2\}$.}

    \centering
    \includegraphics[width=10.5cm]{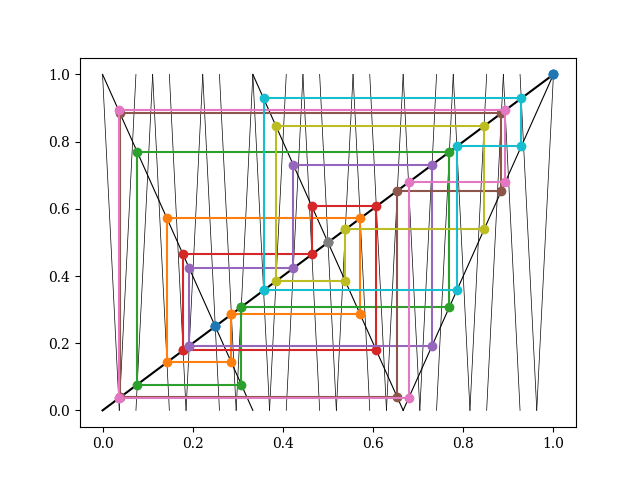}
    
    \caption{Fixed points of $g_{3^3,I}$ for $I=\{2\}$.}
\end{figure}
\end{example}


\newpage
\newpage
\bibliography{mybibtex}{}
\bibliographystyle{plain}

\end{document}